\documentclass[12pt]{amsart}
\usepackage{mathrsfs}
\usepackage[a4paper,asymmetric]{geometry}
\usepackage{mathscinet}
\usepackage{latexsym}
\usepackage{amsthm}
\usepackage{amssymb}
\usepackage{amsfonts}
\usepackage{amsmath}
\newtheorem{theorem}{Theorem}[section]
\newtheorem{thm}[theorem]{Theorem}

\newtheorem{lem}[theorem]{Lemma}
\newtheorem{proposition}[theorem]{Proposition}

\newtheorem{corollary}[theorem]{Corollary}
\newtheorem{assumption}[theorem]{Assumption}
\theoremstyle{definition}

\newtheorem{defn}[theorem]{Definition}

\theoremstyle{remark}

\newtheorem{rem}[theorem]{Remark}
\numberwithin{equation}{section}

 \DeclareMathAlphabet{\mathpzc}{OT1}{pzc}{m}{it}
 
 \newcommand{\E}{\mathbb{E}}            
 \newcommand{\T}{\mathbb{T}}
 \newcommand{\e}{\varepsilon}
 \newcommand{\p}{\partial}

 \newcommand{\Ll}{\langle}
 \newcommand{\Rr}{\rangle}

 \newcommand{\N}{\mathbb{N}}
 \newcommand{\R}{\mathbb{R}}
 \newcommand{\Z}{\mathbb{Z}}
 \newcommand{\PP}{\mathbb{P}}
 \newcommand{\mcl}{\mathcal}
 
 \newcommand{\Be}{\begin{equation}}
 \newcommand{\Ee}{\end{equation}}
 \newcommand{\Bs}{\begin{split}}
 \newcommand{\Es}{\end{split}}
  \newcommand{\Bes}{\begin{equation*}}
 \newcommand{\Ees}{\end{equation*}}
 \newcommand{\BT}{\begin{thm}}
 \newcommand{\ET}{\end{thm}}
 \newcommand{\Bp}{\begin{proof}}
 \newcommand{\Ep}{\end{proof}}
 \newcommand{\BL}{\begin{lem}}
 \newcommand{\EL}{\end{lem}}
 \newcommand{\BP}{\begin{proposition}}
 \newcommand{\EP}{\end{proposition}}
 \newcommand{\BC}{\begin{corollary}}
 \newcommand{\EC}{\end{corollary}}
 \newcommand{\BR}{\begin{rem}}
 \newcommand{\ER}{\end{rem}}
 \newcommand{\BD}{\begin{defn}}
 \newcommand{\ED}{\end{defn}}
 \newcommand{\BI}{\begin{itemize}}
 \newcommand{\EI}{\end{itemize}}
 \newcommand{\eqn}{equation}
 \newcommand{\tl}{\tilde}
 \newcommand{\lt}{\left}
 \newcommand{\rt}{\right}
  \newcommand{\lf}{\left}
 \newcommand{\rg}{\right}

\begin{document}
\title[Applications of technique to convolution]
{Applications of a simple but useful technique to stochastic convolution of $\alpha$-stable processes}
\author[L. Xu]{Lihu Xu}
\address{Department of Mathematics, Brunel University,
Kingston Lane,
Uxbridge,
Middlesex UB8 3PH, United Kingdom}
\email{Lihu.Xu@brunel.ac.uk}
\begin{abstract} \label{abstract}
Our simple but useful technique is using an integration by parts to split the
stochastic convolution into two terms. We develop five applications for this technique.
The first one is getting a uniform estimate of stochastic convolution of $\alpha$-stable processes.
Since $\alpha$-stable noises only have $p<\alpha$ moment, unlike the stochastic convolution of
Wiener process, the well known Da Prato-Kwapie\'n-Zabczyk's factorization (\cite{DPKZ87})
is not applicable. Alternatively, combining this technique with Doob's martingale inequality,
we obtain a uniform estimate similar to
that of stochastic convolution of Wiener process.
Using this estimate, we show that the stochastic convolution of $\alpha$-stable noises stays, with positive probability, in arbitrary small ball with zero center. These two results are important for studying
ergodicity and regularity of stochastic PDEs forced by $\alpha$-stable noises (\cite{DoXuZh12}).
The third application is getting the same results as in \cite{LiuZha12}.
The fourth one gives the trajectory regularity of stochastic Burgers equation forced by $\alpha$-stable noises (\cite{DoXuZh12}). Finally, applying a similar integration by parts to stochastic
convolution of Wiener noises, we get the uniform estimate and continuity property
originally obtained by Da Prato-Kwapie\'n-Zabczyk's factorization.
\end{abstract}

\maketitle

\section{Introduction}
In the study of stochastic PDEs (SPDEs) forced by Wiener noises, we often need to
use a uniform estimate of the stochastic convolution as following
\Be \label{e:WieOU}
\E \sup_{0 \le t \le T} \left\|\int_0^t e^{-A(t-s)} dW_t\right\|^{p} \le C,
\Ee
where $p>0$, $C$ only depends on $p$ and $W_t$ is a $Q$-Wiener process (see \cite[Chapter 4]{DPZ92}).
The estimate \eqref{e:WieOU} has many applications such as ergodicity (\cite{DPZ96}, \cite{DPZ92}, \cite{DO06}) and
regularity of the solutions (\cite{DPZ92}, \cite{DPZ96}). \\

To show \eqref{e:WieOU}, we usually use the well known Da Prato-Kwapie\'n-Zabczyk's factorization technique (\cite{DPZ92}).
This technique requires that the noises has some $p>2$ moments. For an $\alpha$-stable
process ($0<\alpha<2$), it only has $p<\alpha$ moment. Therefore, the factorization technique is not applicable
in this case. \\

Alternatively, we use an integration by parts technique to split
the convolution into two parts and then use Doob's martingale inequality to obtain the uniform estimates.
Applying this uniform estimate, we show that the stochastic convolution of $\alpha$-stable noises stays,
with positive probability, in arbitrary small ball with zero center. These two results are new and important for studying the ergodicity and regularity of SPDEs forced by $\alpha$-stable noises (\cite{DoXuZh12}).
\\

This integration by parts technique has several other applications.
The third and fourth applications are studying the trajectory property of stochastic convolution and stochastic Burgers equations.
We, in particular, get the same results as
in \cite{LiuZha12}. Finally, applying a similar integration by parts to stochastic
convolution of Wiener noises, we get a uniform estimate and continuity property.
\\

Finally, we stress that this simple integration by parts technique has been used in used in the book
\cite[Section 9.4.5]{PeZa07} to study the stochastic convolution of
square integrable L\'evy noises. For the further study of stochastic systems forced by
stable processes, we refer to \cite{ChFu12, PeZa07, BBC04,XZ09,XZ10,PZ09,PXZ10,PSXZ11,Pr10},....
\\

{\bf Acknowledgements:} The author would like to gratefully thank Zhen-Qing Chen, Zhao Dong, 
Yong Liu, Jerzy Zabczyk, Jianliang Zhai and Xicheng Zhang for very helpful discussions. \\

\subsection{Notations and Assumptions.} Let $H$ be a separable Hilbert space. Let $A$ be a self-adjoint operator with
discrete spectrum $\{\gamma_k\}_{k \ge 0}$ satisfying
$$0<\gamma_1 \le \gamma_2 \le... \le \gamma_n \le ..., \ \ \ \ \lim_{n \rightarrow \infty} \gamma_n=\infty.$$
Assume that the eigenvectors $\{e_k\}_{k \ge 0}$ of $A$, i.e., $Ae_k=\gamma_k e_k$ $(k \ge 1)$, from a basis of $H$.
Without loss of generality, we
assume that $\|e_k\|_H=1$ $(k \ge 1)$. For all $x \in H$, it can be uniquely represented by
$$x=\sum_{k \ge 1} x_k e_k,$$
where $x_k \in \R$ for all $k \ge 1$ and $\sum_{k \ge 1} x^2_k<\infty$. \\

Given a $\sigma \in \R$, we can define $A^{\sigma}$ by
$$A^{\sigma}=\sum_{k \ge 1} \gamma^\sigma_k e_k \otimes e_k.$$
where $e_k \otimes e_k: H \rightarrow H$ is a linear operator defined by
$e_k \otimes e_k x=\Ll e_k,x\Rr_H e_k$ for all $x \in H$.

Let $L_t=\sum_{k \ge 1} \beta_k l_k(t) e_k$ be the cylindrical $\alpha$-stable processes on $H$ such that
$\{l_k(t)\}_{k\ge 1}$ are i.i.d. 1 dimensional standard symmetric $\alpha$-stable process sequence, and that
$\{\beta_k\}_{k \ge 1}$ satisfies Assumption \ref{a} below. A
one dimensional standard symmetric $\alpha$-stable process $(l(t))_{t \ge 0}$ has the following
characteristic function
\Be
\E[e^{i \lambda l(t)}]=e^{-t |\lambda|^\alpha}.
\Ee

\begin{assumption} \label{a}
$(\beta_k)_{k \ge 1}$ satisfies $\sum_{k \ge 1} |\beta_k|^\alpha<\infty$.
\end{assumption}
\begin{lem}
Under Assumption \ref{a}, for all $t \ge 0$, $L_t \in H$ a.s..
\end{lem}
\begin{proof}
It is a straightforward corollary of \cite[Proposition 3.3]{PZ09}.
\end{proof}
\ \ \

\subsection{Stochastic convolution and integration by parts technique}
\noindent Consider the following stochastic convolution:
\Be \label{e:OUAlp}
Z(t)=\int_0^t e^{-A(t-s)} d L_s=\sum_{k \in \Z_{*}} z_{k}(t) e_k
\Ee
where $$z_{k}(t)=\int_0^t e^{-\gamma_k (t-s)}
\beta_k d l_k(s).$$
It follows from Ito's product formula (\cite[Theorem 4.4.13]{Ap04}) that for all $k \ge 1$,
\Bes
l_k(t)=\int_0^t \gamma_k e^{-\gamma_k(t-s)} l_k(s)ds+\int_0^t e^{-\gamma_k(t-s)} dl_k(s)+\int_0^t \gamma_k e^{-\gamma_k(t-s)} \Delta l_k(s) ds
\Ees
where $\Delta l_k(s)=l_k(s)-l_k(s-)$. Since $l_k(t)$ is an $\alpha$-stable process, $\Delta l_k(s)=0$ for $s \in [0,t]$ a.s. and thus
\Bes
\int_0^t \gamma_k e^{-\gamma_k(t-s)} \Delta l_k(s) ds=0.
\Ees
Therefore,
\Be \label{e:zkIBP}
\begin{split}
z_k(t)&=\beta_k l_k(t)-\int_0^t \gamma_k e^{-\gamma_k(t-s)} \beta_k l_k(s)ds
\end{split}
\Ee
Hence, we get
\Be \label{e:IntByPar}
\begin{split}
Z(t)=L_t-Y(t),
\end{split}
\Ee
where
$$Y(t):=\int_0^t A e^{-A(t-s)} L_s ds \in H.$$
Since $Z(t)$ and $L_t$ are both in $H$ a.s. for all $t \ge 0$,
we have
$$Y(t) \in H \ \ a.s. \ \ \ \ (t \ge 0).$$
We can further show that $(Y(t))_{t \ge 0}$ is continuous in $H$ a.s. in
Lemma \ref{l:ConY} below.

\section{Main results: two applications of \eqref{e:IntByPar}}
The main results of this paper are the two theorems below. These two theorems are new and
are important for studying the ergodicity and regularity of SPDEs (\cite{DoXuZh12})

\subsection{Application 1} The first application is a uniform estimate of $Z(t)$, which is similar to
that of stochastic convolution of Wiener noises got by Da Prato-Kwapie\'n-Zabczyk's factorization.
This type of estimate is often used to study ergodicity and regularity of stochastic systems (\cite{DoXuZh12, DPKZ87}).
\begin{thm} \label{t:ZEst}
Let $\alpha>1$. Further assume that there exists some $\theta>0$ so that
\Be \label{a2}
\sum_{k \ge 1} |\beta_k|^\alpha \gamma_k^{\alpha \theta}<\infty.
\Ee
Then, for all $0<p<\alpha$, $\tl \theta \in [0,\theta)$ and $T>0$, we have
\Be
\E \sup_{0 \leq t \le T}\|A^{\tl \theta} Z(t)\|^p_{H} \le C T^{p/\alpha} \vee T^{(1/\alpha+\theta-\tl \theta)p} \vee T^{(1+1/\alpha)p}
\Ee
where $C$ depends on $\alpha, \beta, \tl \theta$ and $p$.
\end{thm}

\begin{proof}
We only need to show the inequality for the case $p \in (1,\alpha)$ since the
case of $0<p \le 1$ is an immediate corollary by H$\ddot{o}$lder's inequality.
\\

\emph{Step 1.} We claim that for all $p \in (1,\alpha)$ and all $\tl \theta \in [0,\theta]$,
\Be \label{e:EALT}
\E\|A^{\tl \theta} L_t\|_H^p \le C(\alpha, p) \left(\sum_{k \ge 1} |\beta_k|^\alpha \gamma_k^{\alpha \tl \theta}\rt)^{p/\alpha} t^{p/\alpha} \ \ \ t>0.
\Ee
 To show \eqref{e:EALT}, we only show the inequality for the case $\tl \theta=\theta$ since
the other cases are by the same arguments. \\

 We follow the argument in the proof of \cite[Theorem 4.4]{PZ09}. Take a Rademacher
sequence $\{r_k\}_{k \geq 1}$ in a new probability space
$(\Omega^{'},\mcl F^{'},\PP^{'})$, i.e. $\{r_k\}_{k \geq 1}$ are
i.i.d. with $\PP\{r_k=1\}=\PP\{r_k=-1\}=\frac 12$. By the following
Khintchine inequality: for any $p>0$, there exists some $C(p)>0$
such that for arbitrary real sequence $\{h_k\}_{k \geq 1}$,
$$\left(\sum_{k \geq 1} h^2_k\right)^{1/2} \leq C(p) \left(\E^{'} \left|\sum_{k \geq 1} r_k h_k\right|^p\right)^{1/p}.$$
By this inequality, we get
\begin{\eqn} \label{e:EZAtp}
\begin{split}
\E\|A^{\theta} L_t\|^p_H&=\E \left(\sum_{k \ge 1} \gamma_k^{2\theta}
|\beta_k|^2|l_k(t)|^2\right)^{p/2} \leq C \E \E^{'}\left|\sum_{k \ge 1} r_k \gamma_k^{\theta}
|\beta_k| l_k(t)\right|^p \\
&=C\E^{'}\E\left|\sum_{k \ge 1} r_k \gamma_k^{\theta}
|\beta_k| l_k(t)\right|^p
\end{split}
\end{\eqn}
where $C=C^p(p)$. For any $\lambda \in \R$, by the fact of $|r_k|=1$
and formula (4.7) of \cite{PZ09}, one has
\begin{\eqn*}
\begin{split}
\E\exp \left\{i \lambda \sum_{k \ge 1} r_k \gamma_k^{\theta}
|\beta_k| l_k(t)\right\}&=\exp\left\{-|\lambda|^\alpha \sum_{k \ge 1}
|\beta_k|^\alpha \gamma_k^{\alpha \theta}t
 \right\}
\end{split}
\end{\eqn*}
\vskip 3mm
Now we use (3.2) in \cite{PZ09}: if $X$ is a symmetric random
variable satisfying $$\E \left[e^{i\lambda
X}\right]=e^{-\sigma^{\alpha} |\lambda|^\alpha}$$ for some $\alpha
\in (0,2)$ and any $\lambda \in \R$, then for all $p \in (0,\alpha)$,
$$\E|X|^p=C(\alpha,p)
\sigma^p. $$
Since $\sum_{k \ge 1}
|\beta_k|^\alpha \gamma_k^{\alpha \theta}<\infty$, \eqref{e:EALT} holds. \\

\emph{Step 2}.
Thanks to Step 1, it is easy to get that $A^{\tl \theta} L_t$ ($\tl \theta \in [0,\theta]$) is an $L^p$ martingale.
By Doob's martingale inequality and \eqref{e:EALT},
\Be \label{e:DooIne}
\E\lt[\sup_{0 \leq t \le T} \lf \|A^{\tl \theta} L_t\rg \|^p_H \rt] \le \lt(\frac p{p-1}\rt)^p \E\lt[\lf \|A^{\tl \theta} L_T\rg \|^p_H \rt]
\le C T^{p/\alpha},
\Ee
where $C$ depends on $\tl \theta, \alpha, \beta$ and $p$.
Recall
$$Y(t)=\int_0^t A e^{-A(t-s)} L_s ds,$$
observe that
\Bes
\begin{split}
\sup_{0 \le t \le T}\|A^{\tl\theta} Y(t)\|_H  & \le \sup_{0 \le t \le T} \int_0^t \|A^{1+\tl \theta-\theta} e^{-A(t-s)}\|\|A^{\theta} L_s\|_H ds  \\
& \le \sup_{0 \le t \le T}  \|A^{\theta} L_t\|_H \sup_{0 \le t \le T} \int_0^t \|A^{1+\tl \theta-\theta} e^{-A(t-s)}\| ds
\end{split}
\Ees
By the classical estimate
\Be \label{e:EAtEst}
\|A^{\gamma} e^{-At}\| \le Ct^{-\gamma} \ \ \ \ \ \gamma \ge 0,
\Ee
we get
\Be
\sup_{0 \le t \le T}\|A^{\tl\theta} Y(t)\|_H \le C \sup_{0 \le t \le T}  \|A^{\theta} L_t\|_H (T^{\theta-\tl \theta} \vee T),
\Ee
which, together with Doob's inequality and \eqref{e:EALT}, implies
\Be \label{e:IEst}
\E \sup_{0 \le t \le T}\|A^{\tl\theta} Y(t)\|^p_H \le C T^{(1/\alpha+\theta-\tl \theta)p} \vee T^{(1+1/\alpha)p},
\Ee
where $C$ depends on $\alpha, \beta, \theta, \tl \theta$ and $p$. \\

Combining the above inequality with \eqref{e:DooIne} and \eqref{e:IntByPar}, we immediately get the desired inequality.
\end{proof}
\ \ \
\subsection{Application 2}
Theorem \ref{l:Acc} below is an application of Theorem \ref{t:ZEst}, it of course has its own interest and can be applied to prove
ergodicity (\cite{DoXuZh12}). If $L_t$ is cylindrical Wiener noises,
\eqref{e:Acc} is well known. If $L_t$ is finite dimension $\alpha$-stable noises,
thanks to \cite[Proposition 3, Chapter VIII]{Ber96} and an argument for $I_2$ in the proof below,
we can also obtain \eqref{e:Acc}. However, as $L_t$ is infinite dimensional, Theorem \ref{t:ZEst} is crucial for
passing the limit from the finite to infinite dimensions.

\begin{thm} \label{l:Acc}
Assume that the conditions in Theorem \ref{t:ZEst} hold.
Let $T>0$ and $\e>0$ be arbitrary. For all $\tl \theta \in [0, \theta)$,  we have
\Be \label{e:Acc}
\PP(\sup_{0 \le t\le T} \|A^{\tl \theta} Z(t)\|_H \le \e)>0.
\Ee
\end{thm}

\begin{proof}
Since $\{z_k(t)\}_{k \ge 0}$ are independent sequence, we have
\Bes
\begin{split}
\PP\left(\sup_{0 \le t \le T} \|A^{\tl \theta}Z(t)\|_H \le \e \right) &=
\PP \left(\sup_{0 \le t\le T} \sum_{k \ge 1}\gamma_k^{2\tl \theta} |z_k(t)|^2 \le \e^2\right) \ge I_1I_2
\end{split}
\Ees
where
$$I_1:=\PP \lt(\sup_{0 \le t\le T} \sum_{k>N} \gamma_k^{2\tl \theta} |z_k(t)|^2 \le \e^2/2\rt),$$
$$I_2:=\PP \lt(\sup_{0 \le t\le T} \sum_{k \le N} \gamma_k^{2\tl \theta} |z_k(t)|^2 \le \e^2/2\rt),$$
with $N\in \N$ being some fixed large number.
By the spectral property of $A$, we have
\Bes
\begin{split}
I_1 & \ge \PP \lt(\sup_{0 \le t\le T} \sum_{k>N}\gamma_k^{2\tl \theta} |z_k(t)|^2 \le \e^2/2\rt) \\
& \ge \PP \lt(\sup_{0 \le t\le T} \|A^\theta Z(t)\|^2_H \le \gamma_N^{2(\theta-\tl \theta)} \e^2/2\rt) \\
&=1-\PP \lt(\sup_{0 \le t\le T} \|A^\theta Z(t)\|^2_H>\gamma_N^{2(\theta-\tl \theta)} \e^2/2\rt) \\
&=1-\PP \lt(\sup_{0 \le t\le T} \|A^\theta Z(t)\|^p_H>\gamma_N^{p(\theta-\tl \theta)} \e^p/2^{p/2}\rt)
\end{split}
\Ees
This, together with Theorem \ref{t:ZEst} and Chebyshev inequality, implies
\Bes
I_1 \ge 1-C\gamma_N^{-(\theta-\tl \theta)p} \e^{-p},
\Ees
where $p \in (1, \alpha)$ and $C$ depends on $p, \alpha, \beta, T$. As $\gamma_N$ is sufficient large,
$$I_1>0.$$
\ \ \

To finish the proof, it suffices to show that
\Be \label{e:IrrI2>0}
I_2>0.
\Ee
Define $A_k:=\{\sup_{0 \le t\le T} |z_k(t)| \le \e/(\sqrt {2N} \gamma^{\tl \theta}_k)\}$,
it is easy to have
\Be \label{e:IrrI2Cal}
I_2  \ge \PP\lt(\bigcap_{|k| \le N} A_k\rt)=\prod_{|k| \le N} \PP(A_k).
\Ee
Recalling \eqref{e:zkIBP}, we have
\Bes
z_k(t)=\beta_k l_k(t)-\int_0^t \gamma_k e^{-\gamma_k (t-s)} \beta_k l_k(s) ds.
\Ees
Furthermore, it follows from a straightforward calculation that
\Bes
\sup_{0 \le t \le T} |\int_0^t \gamma_k e^{-\gamma_k(t-s)} \beta_k l_k(s) ds| \le |\beta_k| \sup_{0 \le t \le T} |l_k(t)|
\ \ \ \ \ \ \ k \ge 1.
\Ees
Therefore,
\Bes
\PP(A_k) \ge \PP \lt(\sup_{0 \le t\le T} |l_k(t)| \le \frac{\e}{2 |\beta_k| \sqrt{2N} \gamma^{\tl \theta}_k}\rt)
\Ees
By \cite[Proposition 3, Chapter VIII]{Ber96}, there exist some
$c, C>0$ only depending on $\alpha$ so that
\Bes
\PP(\sup_{0 \le t \le T} |l(t)| \le 1) \ge C e^{-ct}.
\Ees
This, together with the scaling property of stable process, implies
\Be
\PP(A_k)>0 \ \ \ \ |k| \le N,
\Ee
which, combining with \eqref{e:IrrI2Cal}, immediately implies \eqref{e:IrrI2>0}.
\end{proof}
\ \ \ \

\section{Some further applications}
In this section, let us give other three applications of \eqref{e:IntByPar}.
The results in Applications 3 and 5 are known, it seems that we give new and more
illustrative proofs by
our simple technique. The result in Application 4 is also new and from \cite{DoXuZh12}.

\subsection{Application 3}
The third application is to determine the trajectory property
of $(Z(t))_{t \ge 0}$. The theorem implies the results in
\cite{LiuZha12}. \\

For a stochastic process $(X_t)_{t \ge 0}$ valued in some Banach space, it is said to be C$\grave{a}$dl$\grave{a}$g
if it has left limit and is right continuous almost surely.
\begin{lem} \label{l:ConY}
Let $\theta \in \R$ and let the following assumption hold:
\Bes
\sum_{k \ge 1} |\beta_k|^\alpha \gamma_k^{\alpha \theta}<\infty.
\Ees
Then, $(Y(t))_{t \ge 0}$ is continuous in $\mcl D(A^\theta)$ a.s..
\end{lem}

\begin{proof}
It is easy to check that under the condition in the lemma, $Z_t$ and
$L_t$ are both in $\mcl D(A^\theta)$ for all $t \ge 0$. Thanks to \eqref{e:IntByPar},
$Y(t) \in \mcl D(A^\theta)$ for all $t \ge 0$. We only prove the lemma for the case of $\theta=0$
since the other cases are by the same arguments.  \\

By \cite[Theorem 4.13]{PeZa07}, for all $T>0$,
$L_t$ uniformly converges in $H$ on $[0,T]$ a.s.. This, together with \cite[Lemma 3.1]{LiuZha12},
implies that $L_t$ has a C$\grave{a}$dl$\grave{a}$g version on $[0,T]$. Since $T>0$ is arbitrary,
$(L_t)_{t\ge 0}$ has a C$\grave{a}$dl$\grave{a}$g version.
By \cite[Theorem 9.3]{PeZa07}, $(Z(t))_{t \ge 0}$ has a C$\grave{a}$dl$\grave{a}$g modification.
Hence, $(Y(t))_{t \ge 0}$ also has a C$\grave{a}$dl$\grave{a}$g version.  \\

Write $Y(t)=\sum_{k=1}^\infty y_k(t) e_k$ with
$$y_k(t):=\int_0^t \gamma_k e^{-\gamma_k(t-s)} \beta_k l_k(s)ds,$$
further denote
$$Y_{n}(t):=\sum_{k<n} y_k(t) e_k, \ \ \ Y^n(t):=\sum_{k \ge n} y_k(t) e_k,  \ \ \ \ \ (n \in \N).$$
Since $(Y(t))_{0 \le t \le T}$ also has a C$\grave{a}$dl$\grave{a}$g version for all
$T>0$, by \cite[Lemma 3.1]{LiuZha12}, we get
\Bes
\lim_{n \rightarrow \infty} \sup_{0 \le t \le T} \|Y^n(t)\|^2_H=0 \ \ \ with \ probability \ 1.
\Ees
On the other hand, $(Y_n(t))_{t \ge 0}$ is continuous for all $n \ge 0$. By \cite[Lemma 3.1]{LiuZha12}
again, $(Y(t))_{0 \le t\le T}$ is continuous on $H$. Since $T>0$ is arbitrary, $(Y(t))_{t\ge 0}$ is continuous on $H$.
\end{proof}

\begin{thm}
Under the same condition as in Lemma \ref{l:ConY},
$(Z(t))_{t \ge 0}$ has a C$\grave{a}$dl$\grave{a}$g version
in $\mcl D(A^{\theta})$ iff $(L_t)_{t\ge 0}$ has
a C$\grave{a}$dl$\grave{a}$g version in $\mcl D(A^{\theta})$.
\end{thm}

\begin{proof}
Using the integration by formula \eqref{e:IntByPar}, we have
$$Z(t)=L_t-Y(t).$$
Since $(Y(t))_{t \ge 0}$ is continuous in $\mcl D(A^{\theta})$,
 $(Z(t))_{t \ge 0}$ has a C$\grave{a}$dl$\grave{a}$g version iff $(L_t)_{t\ge 0}$ has
a C$\grave{a}$dl$\grave{a}$g version.
\end{proof}
\begin{rem}
Using above theorem, we can easily recover the results in \cite{LiuZha12}, but our proof seems much more illustrative.
\end{rem}
\ \ \

\subsection{Application 4} The fourth application of Theorem \ref{t:ZEst} is to determine the trajectory
property of stochastic Burgers equations:
\Be \label{e1.1}
d X(t)-\nu \p^2_\xi X(t)dt+X(t) \p_\xi X(t)dt=dL_t, \ \ \ X(0)=x.
\Ee
\ \ \ \

Let $\T= \R/(2\pi\Z)$ be equipped with the usual Riemannian metric, and let $d \xi$
denote the Lebesgue measure on $\T$. Then
$$H:=\bigg\{x\in L^2(\T, \R): \int_\T x(\xi) d\xi =0\bigg\}$$
is a separable real Hilbert space with  inner product and norm
$$\Ll x,y \Rr_H:=\int_\T x(\xi)y(\xi) d\xi,\ \ \|x\|_H:=\Ll x,x \Rr_H^{1/2}.$$
\ \ \ \

For $x\in C^2(\T)$, the Laplacian operator $\Delta$ is given by $\Delta x= x''.$
Let $(A, \mcl D(A))$ be the closure of $(-\Delta, C^2(\T)\cap H)$ in $H$,
which is a positively definite self-adjoint operator on $H$.
Denote $\Z_*:=\Z \setminus \{0\}$. $\{e_k\}_{k \in \Z_*}$ with $e_k=\frac{1}{\sqrt{2 \pi}} e^{ik\xi}$
an orthonormal basis of $H$. It is easy to see that
$$A e_k=|k|^2 e_k.$$
\ \ \ \

Assume that $L_t=\sum_{k \in \Z_*} \beta_k l_k(t)$ is the cylindrical $\alpha$-stable processes on $H$ with
$\{l_k(t)\}_{k\in \Z_*}$ being i.i.d. standard 1 dimensional $\alpha$-stable process sequence. Moreover, there are some
constants $C_1, C_2>0$ so that
$C_1 |k|^{-2\beta} \le |\beta_k| \leq C_2
|k|^{-2\beta}$ with $\beta>1+\frac 1{2\alpha}$.
\\

Under the above setting, Theorems \ref{t:ZEst} and \ref{l:Acc} reads as
the following two theorems respectively:
\begin{thm} \label{l:ZEstApp}
Let $\alpha>1$, $0<\theta<\beta-\frac 1{2 \alpha}$ and $T>0$ be all arbitrary. For all $0<p<\alpha$ and
$\tl \theta \in [0,\theta)$, we have
\Be
\E \sup_{0 \leq t \le T}\|A^{\tl \theta} Z(t)\|^p_{H}<C
\Ee
where $C$ depends on $\alpha,\theta,p,T$.
\end{thm}
\begin{thm} \label{t:AccApp}
Under the same conditions as in Theorem \ref{l:ZEstApp}. For all $T>0$, $\tl \theta \in [0,\theta)$ and $\e>0$, we have
\Bes
\PP(\sup_{0 \le t\le T} \|A^{\tl \theta} Z(t)\|_H \le \e)>0.
\Ees
\end{thm}

The following result is from \cite[Theorem 2.1]{DoXuZh12}, Theorem \ref{l:ZEstApp} plays a crucial
role in its proof (\cite[Section 4]{DoXuZh12}).
\begin{thm}
For all $x \in H$,
Eq. \eqref{e1.1} admits a unique mild solution with C$\grave{a}$dl$\grave{a}$g trajectory.
\end{thm}
\ \ \ \

\subsection{Application 5.} Let us use the integration by parts technique to study the stochastic convolution of Wiener noises.
Let $Q$ be Hilbert-Schmidt in $H$, i.e., $Q$ is a linear operator such that
\Be
\|Q\|_{HS}^2:=\sum_{k \ge 1} \|Qe_k\|^2_H<\infty.
\Ee
\begin{thm}
Let $(W_t)_{t \ge 0}$ be a cylindrical white Wiener noise in $H$, i.e., $W_t$ can be formally written as
$W_t=\sum_{k \ge 1} w_k(t) e_k$ (\cite[pp 48]{DPZ92}) with
$\{w_k(t)\}_{k \ge 1}$ being a sequence of i.i.d. 1d Brownian motions.
Let $(Z_W(t))_{t \ge 0}$ be the stochastic convolution defined by
\Be
Z_W(t)=\int_0^t e^{-A(t-s)} QdW_s.
\Ee
If we further assume that there exists an (arbitrary small) $\theta>0$ so that $\|A^\theta Q\|_{HS}<\infty$, then the following
statements hold:
\begin{enumerate}
\item For all $\tl \theta \in [0, \theta)$, $T>0$ and $p>0$, we have
\Bes
\E \sup_{0 \le t \le T} \|A^{\tl \theta} Z_W(t)\|^p_H \le C T^{3p/2} \vee T^{(1/2+\theta-\tl \theta)p} \vee T^{p/2},
\Ees
where $C$ depends on $\tl \theta$ and $p$.
\item $(Z_W(t))_{t \ge 0}$ is continuous in $\mcl D(A^{\tl \theta})$ with $\tl \theta \in [0,\theta)$.
\end{enumerate}
\end{thm}

\begin{proof}
Applying It$\hat{o}$ formula (\cite{PrClRo07}) to $e^{-A(t-s)}Q W_s$, we get
\Be
Z_W(t)=Q W_t-\int_0^t A e^{-A(t-s)} Q W_s ds.
\Ee
This immediately implies (2).

Now let us show (1). Since the
case of $0<p<2$ immediately follows from H$\ddot{o}$lder inequality and the
case of $p \ge 2$, we only need to show the inequality for the case $p \ge 2$.
Since $A^{\theta} Q$ is Hilbert-Schmidt, by Doob's martingale inequality, we have
\Be
\begin{split}
\E \sup_{0 \le t \le T} \|A^{\tl \theta} QW_t\|^p_{H} \le \lt(\frac{p}{p-1}\rt)^p \E  \|A^{\tl \theta} QW_T\|^p_{H}
\le C T^{p/2},
\end{split}
\Ee
where $C$ depends on $p$ and $\|A^{\theta} Q\|_{HS}$, and the last inequality is by Fernique theorem for Gaussian measure (\cite{Hai09}).
\\

Denote
$$I:=\int_0^t A e^{-A(t-s)} Q W_s ds,$$
Observe
\Be
\begin{split}
\|A^{\tl \theta} I\|_H  \le \int_0^t \|A^{1-(\theta-\tl \theta)} e^{-A(t-s)}\| \sup_{0 \le s \le T}\|A^\theta Q W_s\| ds.
\end{split}
\Ee
This, together with \eqref{e:EAtEst} and Doob's inequality, implies
\Be
\begin{split}
\E \sup_{0 \le t \le T} \|A^{\tl \theta} I\|^p_H &\le \E \sup_{0 \le t \le T} \|A^\theta Q W_t\|_H^p \sup_{0 \le t \le T} \lt|
\int_0^t (t-s)^{-1+(\theta-\tl \theta)} \vee 1 ds\rt|^p  \\
& \le C T^{3p/2} \vee T^{(1/2+\theta-\tl \theta)p}
\end{split}
\Ee
where $C$ depends on $p$ and $\|A^{\theta} Q\|_{HS}$. \\

Combining the two estimates, we immediately get the desired inequality in (1).
\end{proof}


\bibliographystyle{amsplain}

\end{document}